\documentclass[12pt]{amsart}
\usepackage{amsmath, amsfonts, amssymb, amsthm,hyperref,mathtools,array}
\hypersetup{hypertex=true,
	colorlinks=true,
	linkcolor=blue,
	anchorcolor=blue,
	citecolor=blue}
\usepackage{bm}
\allowdisplaybreaks[4]
\def\({\bg(}
\def\){\bg)}

\def\ord{{\rm ord}}

\def\Tr{{\rm Tr}}

\def\v{{\bm v}}
\def\u{{\bm u}}

\def\pmod #1{\ ({\rm{mod}}\ #1)}
\def\mod #1{\ {\rm mod}\ #1}
\def\Ack{\medskip\noindent {\bf Acknowledgments}}

\theoremstyle{plain}
\newtheorem{theorem}{Theorem}[section]
\newtheorem{lemma}{Lemma}
\newtheorem{corollary}{Corollary}
\newtheorem{conjecture}{Conjecture}

\theoremstyle{definition}

\theoremstyle{remark}
\newtheorem{remark}{Remark}

\makeatletter
\@namedef{subjclassname@2020}{%
	\textup{2020} Mathematics Subject Classification}
\makeatother
\vspace{4mm}

\begin{document}
	\medskip
	
	\title[The Pell sequence and cyclotomic matrices involving squares over finite fields]
	{The Pell sequence and cyclotomic matrices involving squares over finite fields}
	\author[H.-L. Wu, L.-Y. Wang and H.-X. Ni]{Hai-Liang Wu, Li-Yuan Wang and He-Xia Ni*}
	
	\address {(Hai-Liang Wu) School of Science, Nanjing University of Posts and Telecommunications, Nanjing 210023, People's Republic of China}
	\email{\tt whl.math@smail.nju.edu.cn}
	
	\address {(Li-Yuan Wang) School of Physical and Mathematical Sciences, Nanjing Tech University, Nanjing 211816, People's Republic of China}
	\email{\tt wly@smail.nju.edu.cn}
	
	\address{(He-Xia Ni) Department of applied mathematics, Nanjing Audit University, Nanjing 211815, People's Republic of China}
	\email{\tt nihexia@yeah.net}

	\keywords{the Pell sequence, cyclotomic matrices, character sums over finite fields.
		\newline \indent This research was supported by the Natural Science Foundation of China (Grant Nos. 12101321, 12201291 and 12371004).
	    \newline \indent *Corresponding author}
	
	\begin{abstract}
		In this paper, by some arithmetic properties of the Pell sequence and some $p$-adic tools, we study certain cyclotomic matrices involving squares over finite fields. For example, let $1=s_1,s_2,\cdots,s_{(q-1)/2}$ be all the nonzero squares over $\mathbb{F}_{q}$, where $q=p^f$ is an odd prime power with $q\ge7$. We prove that the matrix 
		$$B_q((q-3)/2)=\left[\left(s_i+s_j\right)^{(q-3)/2}\right]_{2\le i,j\le (q-1)/2}$$
		is a singular matrix whenever $f\ge2$. Also, for the case $q=p$, we show that 
		$$\det B_p((p-3)/2)=0\Leftrightarrow Q_p\equiv 2\pmod{p^2\mathbb{Z}},$$
		where $Q_p$ is the $p$-th term of the companion Pell sequence $\{Q_i\}_{i=0}^{\infty}$ defined by $Q_0=Q_1=2$ and $Q_{i+1}=2Q_i+Q_{i-1}$.
	\end{abstract}
	\maketitle
	
	\tableofcontents

	\section{Introduction}
	\setcounter{lemma}{0}
	\setcounter{theorem}{0}
	\setcounter{equation}{0}
	\setcounter{conjecture}{0}
	\setcounter{remark}{0}
	\setcounter{corollary}{0}
	
	\subsection{Notation} In this paper, the symbol $p$ always denotes an odd prime. We let $\mathbb{Q}_p$ be the $p$-adic number field and let $\mathbb{Z}_p$ be the ring of all $p$-adic integers. As usual, we use  $\mathbb{C}_p$ to denote the completion of the algebraic closure of $\mathbb{Q}_p$. In addition, we let 
	$$\ord_p:\ \mathbb{C}_p\rightarrow \mathbb{R}$$
	be the $p$-adic order function over $\mathbb{C}_p$. 
	
	 Let $q=p^f$ be an odd prime power with $f\in\mathbb{Z}^+$. The symbol $\mathbb{F}_q$ denotes the finite field with $q$ elements and $\mathbb{F}_q^{\times}=\mathbb{F}_q\setminus\{0\}$ is the multiplicative cyclic group of all nonzero elements over $\mathbb{F}_q$. The group of all multiplicative characters of $\mathbb{F}_q$ is denoted by  $\widehat{\mathbb{F}_q^{\times}}$, and the trivial character is written by $\varepsilon$. Throughout this paper, for any multiplicative character $\psi: \mathbb{F}_q^{\times}\rightarrow\mathbb{C}$ (or $\mathbb{C}_p$), we extend $\psi$ to $\mathbb{F}_q$ by letting $\psi(0)=0$.
	
	Let $\zeta_p$ be a primitive $p$-th root of unity over $\mathbb{C}$ (or $\mathbb{C}_p$). Then for any $\psi\in\widehat{\mathbb{F}_q^{\times}}$, the Gauss sum of $\psi$ is defined by 
	$$G_q(\psi)=\sum_{x\in\mathbb{F}_q}\psi(x)\zeta_p^{\Tr_{\mathbb{F}_q/\mathbb{F}_p}(x)},$$
	where $\Tr_{\mathbb{F}_q/\mathbb{F}_p}: \mathbb{F}_q\rightarrow \mathbb{F}_p$ is the trace map. Also, for any $\psi,\chi\in\widehat{\mathbb{F}_q^{\times}}$, the Jacobi sum of $\psi$ and $\chi$ is defined by 
	$$J_q(\psi,\chi)=\sum_{x\in\mathbb{F}_q}\psi(x)\chi(1-x).$$

	For any positive integer $a$ and nonnegative integer $b$, the symbol $\binom{a}{b}$ is the binomial coefficient, that is, 
	$$\binom{a}{b}=\frac{a(a-1)\cdots(a-b+1)}{b!}.$$
	
	Finally, for any square matrix $M$ over a field, $\det M$ denotes the determinant of $M$. 
	
	\subsection{Background and motivation}
	
	The investigations of cyclotomic matrices involving elements over finite fields were initiated by Lehmer \cite{Lehmer} and Carlitz \cite{Carlitz}. For instance, for any nontrivial multiplicative character $\psi$ modulo $p$, Carlitz \cite{Carlitz} first studied the cyclotomic matrices
	\begin{equation}\label{Eq. definition of Carlitz Cp-}
		C_p^-(\psi):=\left[\psi(j-i)\right]_{1\le i,j\le p-1},
	\end{equation}
	and 
	\begin{equation}\label{Eq. definition of Carlitz Cp+}
		C_p^+(\psi):=\left[\psi(j+i)\right]_{1\le i,j\le p-1}.
	\end{equation}
	Carlitz \cite[Theorem 4 and Theorem 5]{Carlitz} proved that 
	\begin{equation*}
		\det C_p^-(\psi)=(-1)^{(p-1)/m}\cdot\frac{G_p(\psi)^{p-1}}{p},
	\end{equation*}
	and
	\begin{equation*}
		\det C_p^+(\psi)=\begin{cases}
			\frac{1}{p}(-1)^{\frac{p-1}{2m}}G_p(\psi)^{p-1} & \mbox{if}\ m\equiv 1\pmod 2,\\
			\frac{1}{p}(-1)^{\frac{p-1}{m}}\delta(\psi)^{p-1}G_p(\psi)^{p-1} & \mbox{if}\ m\equiv 0\pmod 2,
		\end{cases}
	\end{equation*}
	where $m=\min\{k\in\mathbb{Z}^+: \psi^k=\varepsilon\}$ is the order of $\psi$ and 
	\begin{equation*}
		\delta(\psi)=\begin{cases}
			1         & \mbox{if}\ \psi(-1)=1,\\
			-{\bf i}  & \mbox{if}\  \psi(-1)=-1.
		\end{cases}
	\end{equation*}
	
	Along Carlitz's work, Chapman \cite{Chapman} further studied some variants of the matrix $C_p^+(\psi)$. In fact, let $(\frac{\cdot}{p})$ denotes the Legendre symbol, i.e., the unique quadratic multiplicative character of $\mathbb{F}_p$. Chapman considered the matrices
	\begin{equation*}
		C_p^{(0)}:=\left[\left(\frac{i+j}{p}\right)\right]_{0\le i,j\le (p-1)/2},
	\end{equation*} 
	and 
	\begin{equation*}
		C_p^{(1)}:=\left[\left(\frac{i+j}{p}\right)\right]_{1\le i,j\le (p-1)/2}.
	\end{equation*} 
	Although Chapman only changed the sizes of the matrices, the calculations of $\det C_p^{(0)}$ and $\det C_p^{(1)}$ became extremely complicated. In fact, for $p\equiv 1\pmod 4$, let $\epsilon_p>1$ and $h_p$ denote the fundamental unit and class number of $\mathbb{Q}(\sqrt{p})$ respectively, and write 
	$$\epsilon_p^{h_p}=a_p+b_p\sqrt{p}$$
	with $a_p,b_p\in\mathbb{Q}$. Then Chapman \cite{Chapman} showed that 
	$$\det C_p^{(0)}=\begin{cases}
		(-1)^{(p+3)/4}2^{(p-1)/2}a_p   &  \mbox{if}\ p\equiv 1\pmod 4,\\
		-2^{(p-1)/2}                   &  \mbox{if}\ p\equiv 3\pmod 4,
	\end{cases}$$
	and that 
	$$\det C_p^{(1)}=\begin{cases}
		(-1)^{(p-1)/4}2^{(p-1)/2}b_p   &  \mbox{if}\ p\equiv 1\pmod 4,\\
		0                              &  \mbox{if}\ p\equiv 3\pmod 4.
	\end{cases}$$
	
	Chapman also investigated some variants of $C_p^-(\psi)$, and the most well-known variant among them is Chapman's ``evil determinant'':
	$$\det C_p^{-}:=\det \left[\left(\frac{j-i}{p}\right)\right]_{1\le i,j\le (p+1)/2}.$$
	Let 
	$$\epsilon_p^{(2+(-1)^{(p+3)/4})h_p}=a_p'+b_p'\sqrt{p}$$
	with $a_p',b_p'\in\mathbb{Q}$. Then Chapman conjectured that 
	$$\det C_p^-=\begin{cases}
		-a_p' & \mbox{if}\ p\equiv 1\pmod 4,\\
		1     & \mbox{if}\ p\equiv 3\pmod 4.
	\end{cases}$$
	Vsemirnov \cite{Vsemirnov12,Vsemirnov13} later confirmed this tough conjecture by using ingenious and sophisticated matrix decompositions.
	
	Sun \cite{Sun19} studied the variants of the above matrices from another perspective. For example, Sun \cite{Sun19} consider the matrix 
	\begin{equation}\label{Eq. Sun Sp}
		S_p:=\left[\left(\frac{i^2+j^2}{p}\right)\right]_{1\le i,j\le (p-1)/2}
	\end{equation}
	which involves the nonzero squares over $\mathbb{F}_p$. Sun \cite[Theorem 1.2]{Sun19} proved that $-\det S_p \mod {p\mathbb{Z}}$ is a nonzero square over $\mathbb{Z}/p\mathbb{Z}=\mathbb{F}_p$. Moreover, Sun conjectured that $-\det S_p$ is indeed a square of some integer whenever $p\equiv 3\pmod 4$. This conjecture was later confirmed by Alekseyev and Krachun. For the case $p\equiv 1\pmod 4$, if we write $p=a^2+4b^2$ with $a,b\in\mathbb{Z}$ and $a\equiv 1\pmod 4$, then Cohen, Sun and Vsemirnov conjectured that $S_p/a$ is also a square of certain integer. This conjecture was later confirmed by the first author \cite{Wu-CR}.
	
	Recently, for any positive integer $m$, Sun \cite{Sun24} considered the matrix 
	\begin{equation}\label{Eq. Sun 24 matrix}
		S_p(m):=\left[\left(i^2+j^2\right)^m\right]_{1\le i,j\le (p-1)/2},
	\end{equation}
	and obtained several interesting results concerning $\det S_p(m)$. For example, when $p\equiv 3\pmod 4$, Sun \cite[Theorem 1.2(ii)]{Sun24} proved that 
	$$\det S_p(p-3)=\det \left[\frac{1}{(i^2+j^2)^2}\right]_{1\le i,j\le (p-1)/2}\equiv \frac{1}{4}\prod_{r=1}^{(p-3)/4}\left(r+\frac{1}{4}\right)^2\pmod{p\mathbb{Z}_p}.$$
	
	After introducing the above relevant research results, we now state our research motivations. Inspired by Sun's matrix $S_p(m)$ defined by Eq. (\ref{Eq. Sun 24 matrix}), in this paper, we mainly consider a variant of $S_p(m)$. In fact, let $q=2n+1$ and let 
	$$\left\{s_1=1,s_2,\cdots,s_n\right\}=\{x^2: x\in\mathbb{F}_q^{\times}\}$$
	be the set of all nonzero squares over $\mathbb{F}_q$. Then we define the matrix 
	\begin{equation}\label{Eq. our matrix Bq(m)}
		B_q(m):=\left[\left(s_i+s_j\right)^m\right]_{2\le i,j\le n}.
	\end{equation}
   
	Before studying the properties of $B_q(m)$, let's first have a brief discussion.
	Let 
	$$h(t)=a_{m-1}t^{m-1}+\cdots+a_1t+a_0$$
	be a polynomial over a commutative ring with $\deg(h(t))\le m-1$. Then it is known that (see \cite[Lemma 9]{K})
	\begin{equation}\label{Eq. the K determinant formula}
		\det \left[h(x_i+y_j)\right]_{1\le i,j\le m}=a_{m-1}^{m}\cdot\prod_{r=0}^{m-1}\binom{m-1}{r}\cdot\prod_{1\le i<j\le m}\left(x_i-x_j\right)\left(y_j-y_i\right).
	\end{equation}
	By this we immediately see that $B_q(m)$ is singular whenever $m\le n-3$. Hence it is meaningful to consider the cases of $n-2\le m\le q-1$. In this paper, we focus on the cases of $m=n-2,n-1,n$, and we will see later that the methods used in each case are very different.
	
	\subsection{The Pell sequence} Before the statements of our main results, we introduce the Pell sequence. The Pell sequence is an infinite sequence of integers that comprise the denominators of the closest rational approximations to the square root of 2. Specifically, 
	the Pell sequence $\{P_i\}_{i=0}^{\infty}$ is defined by
	\begin{equation*}
		P_0=0,P_1=1, P_{i+1}=2P_i+P_{i-1},
	\end{equation*}
	and its companion sequence $\{Q_i\}_{i=0}^{\infty}$ is defined by 
	\begin{equation*}
		Q_0=Q_1=2, Q_{i+1}=2Q_i+Q_{i-1}.
	\end{equation*}
	
	For any integer $i\ge0$, it is also known that 
	$$P_i=\frac{1}{2\sqrt{2}}\left((1+\sqrt{2})^i-(1-\sqrt{2})^i\right)=\sum_{k=0}^{\lfloor(i-1)/2\rfloor}\binom{i}{2k+1}2^k,$$
	and
	 $$Q_i=(1+\sqrt{2})^i+(1-\sqrt{2})^i.$$
	By this, it is easy to see that 
	$$P_p\equiv 2^{(p-1)/2}\equiv \left(\frac{2}{p}\right)\pmod{p\mathbb{Z}},$$
	and that 
	$$Q_p\equiv 2\pmod {p\mathbb{Z}}.$$

	\subsection{Main results} Now we state our first result, which involves the cases $m=n-1,n-2$.
	
     \begin{theorem}\label{Thm. case m is less than n+1}
     	Let $q=p^f=2n+1\ge 7$ be an odd prime power with $p$ prime and $f\in\mathbb{Z}^+$. Then the following results hold.
     	
     	{\rm (i)} $B_q(n-1)$ is a singular matrix whenever $f\ge2$. Moreover, if $f=1$, then 
     	\begin{equation*}
     		\det B_p(n-1)=\frac{2\cdot((n-1)!)^n}{(0!1!\cdots (n-1)!)^2}\cdot a_p\in\mathbb{F}_p,
     	\end{equation*}
     where 
     \begin{equation*}
     	a_p=\frac{2-Q_p}{p}\mod {p\mathbb{Z}}\in\mathbb{Z}/p\mathbb{Z}=\mathbb{F}_p.
     \end{equation*}
 
     {\rm (ii)} $B_q(n-2)$ is singular if and only if the integer $f\ge2$. When $f=1$, we have  
        \begin{equation*}
        	\det B_p(n-2)=\left(-\frac{1}{2}\right)^{n-2}\cdot\frac{((n-2)!)^{n-1}}{(0!1!\cdots (n-2)!)^2}\in\mathbb{F}_p^{\times}.
        \end{equation*}
     \end{theorem}
	
	By Theorem \ref{Thm. case m is less than n+1}, we obtain the following result.
	
	\begin{corollary}\label{Cor. of Thm. case m is less than n+1}
		Let $p=2n+1\ge7$ be an odd prime. Then the following results hold.
		
		{\rm (i)} $B_p(n-1)$ is singular if and only if 
		\begin{equation*}
			Q_p\equiv 2\pmod {p^2\mathbb{Z}}.
		\end{equation*}
	Moreover, if $p\equiv1\pmod4$ , then 
	\begin{equation*}
		\left(\frac{\det B_p(n-1)}{p}\right)=\left(\frac{2a_p}{p}\right).
	\end{equation*}
     If $p\equiv 3\pmod 4$, then 
     \begin{equation*}
     	\left(\frac{\det B_p(n-1)}{p}\right)=(-1)^{\frac{h(-p)-1}{2}}\left(\frac{a_p}{p}\right),
     \end{equation*}
     where $h(-p)$ is the class number of the imaginary quadratic field $\mathbb{Q}(\sqrt{-p})$. 
     
     {\rm (ii)} Suppose $p\equiv 1\pmod 4$. Then 
     \begin{equation*}
     	\left(\frac{\det B_p(n-2)}{p}\right)=\left(\frac{6}{p}\right).
     \end{equation*}
 
      Suppose $p\equiv 3\pmod 4$. Then 
     \begin{equation*}
     	\left(\frac{\det B_p(n-2)}{p}\right)=\left(\frac{-2}{p}\right).
     \end{equation*}
	\end{corollary}
	
	\begin{remark}
		With the help of a computer and by the referee's comments, we verify that  
		\begin{equation}\label{Eq. verification of Qp}
			\left\{p:\ p\ \text{is a prime with}\ 7\le p\le 10^6\ \text{and}\ Q_p\equiv 2\pmod{p^2\mathbb{Z}}\right\}=\left\{13,31\right\}.
		\end{equation}
		Motivated by this, we pose the following conjecture.
	\end{remark}
	
	\begin{conjecture}
		Let $p\ge7$ be an odd prime. Then $B_p(n-1)$ is singular if and only if $p\in\{13,31\}$. 
	\end{conjecture}
	
    We next turn to the case $m=n$. 
	
	\begin{theorem}\label{Thm. case m=n}
		Let $q=p^f=2n+1\ge 7$ be an odd prime power with $p$ prime and $f\in\mathbb{Z}^+$. Then 
		$B_q(n)$ is singular whenever $f\ge2$. Moreover, if $f=1$, then 
		$$\det B_p(n)=(-1)^{n}\left(\frac{1}{2}\right)^{n-2}\cdot\frac{(n!)^{n+1}}{(0!1!\cdots n!)^2}\cdot b_p\in\mathbb{F}_p,$$
		where 
		$$b_p=\frac{2(\frac{2}{p})-2P_p-p}{p}\mod{p\mathbb{Z}}\in\mathbb{Z}/p\mathbb{Z}=\mathbb{F}_p.$$
	\end{theorem}
	
	Similar to Corollary \ref{Cor. of Thm. case m is less than n+1}, we have the following result.
	
	\begin{corollary}\label{Cor. of Thm. case m=n}
		Let $p=2n+1\ge7$ be a prime. Then 
		\begin{equation*}
			\left(\frac{\det B_p(n)}{p}\right)=\left(\frac{-2b_p}{p}\right),
		\end{equation*}
	and hence $B_p(n)$ is singular if and only if 
	$$2P_p\equiv 2\left(\frac{2}{p}\right)-p\pmod{p^2\mathbb{Z}}.$$
	\end{corollary}
	
	\begin{remark}
		As (\ref{Eq. verification of Qp}), it is natural to consider the odd prime $p$, which satisfies the congruence  
		$$2P_p\equiv 2\left(\frac{2}{p}\right)-p\pmod{p^2\mathbb{Z}}.$$
		By the referee's comments, we have 
		\begin{equation}\label{Eq. verification of Pp}
			\left\{p:\ p\ \text{is a prime with}\ 7\le p\le 10^6\ \text{and}\ 2P_p\equiv 2\left(\frac{2}{p}\right)-p\pmod{p^2\mathbb{Z}}\right\}=\left\{29\right\}.
		\end{equation}
	\end{remark}

	Using essentially the same method appeared in the proof of Theorem \ref{Thm. case m=n}, we can obtain some satisfactory results on certain variants of Carlitz's matrice $C_p^{-}(\psi)$ and $C_p^{+}(\psi)$ defined by Eq. (\ref{Eq. definition of Carlitz Cp-}) and Eq. (\ref{Eq. definition of Carlitz Cp+}) respectively.
	 
	In fact, let $\mathbb{F}_q^{\times}=\{x_1=1,x_2,\cdots,x_{q-1}\}$. For any nontrivial character $\psi\in\widehat{\mathbb{F}_q^{\times}}$, we define 
	\begin{equation*}
		D_q^{-}(\psi):=\left[\psi(x_j-x_i)\right]_{2\le i,j\le q-1},
	\end{equation*}
	and 
	\begin{equation*}
		D_q^{+}(\psi):=\left[\psi(x_j+x_i)\right]_{2\le i,j\le q-1}.
	\end{equation*}
	
	We have the following result.
	
	\begin{theorem}\label{Thm. a variant of Carlitz}
		Let $q=p^f\ge3$ be an odd prime power. Then for any nontrivial character $\psi\in\widehat{\mathbb{F}_q^{\times}}$, we have 
		$$\det D_q^{-}(\psi)=-\frac{1+\psi(-1)}{q^2}G_q(\psi)^{q-1},$$
		and
		$$\det D_q^{+}(\psi)=\frac{(-1)^{(q+1)/2}\cdot\psi(-1)}{q^2}\left(2-\overline{\psi(2)}\right)G_q(\psi)^{q-1}.$$
	\end{theorem}
	
	Let $\psi=\phi$ be the unique quadratic character of $\mathbb{F}_q$, i.e., 
	$$\phi(x)=\begin{cases}
		1  & \mbox{if}\ x\in\{s_1,s_2,\cdots,s_{(q-1)/2}\},\\
		0  & \mbox{if}\ x=0,\\
		-1 & \mbox{otherwise.}
	\end{cases}$$
	Then $D_p^-(\phi)$ is an integer matrices. Note that (see \cite[Corollary 3.7.6]{Cohen})
	$$G_q(\phi)=\begin{cases}
		(-1)^{f-1}\sqrt{q}           & \mbox{if}\ p\equiv 1\pmod 4,\\
		(-1)^{f-1}{\bf i}^f\sqrt{q}  & \mbox{if}\ p\equiv 3\pmod 4.\\
		\end{cases}$$
	By this and Theorem \ref{Thm. a variant of Carlitz}, we can obtain the following corollary directly.
	
	\begin{corollary}\label{Cor. of Thm. a variant of Carlitz}
		Let $q=p^f\ge3$ be an odd prime power. Then $D_q^-(\phi)$ is singular if and only if $q\equiv 3\pmod4$. For $q\equiv 1\pmod 4$, we have 
		$$\det D_q^-(\phi)=-2q^{(q-5)/2}.$$
	\end{corollary}
	
	\subsection{Outline of this paper} We will prove Theorem \ref{Thm. case m is less than n+1} and its corollary in Section 2. In Section 3, we shall introduce some necessary results on almost circulant matrices and some $p$-adic tools. The proofs of Theorems \ref{Thm. case m=n} and \ref{Thm. a variant of Carlitz} will be given in Section 4 and Section 5 respectively.

	\section{Proofs of Theorem \ref{Thm. case m is less than n+1} and its corollary}
	\setcounter{lemma}{0}
	\setcounter{theorem}{0}
	\setcounter{equation}{0}
	\setcounter{conjecture}{0}
	\setcounter{remark}{0}
	\setcounter{corollary}{0}
	
	Recall that $n=(q-1)/2$ and $s_1=1,s_2,\cdots,s_n$ are exactly all the nonzero squares over $\mathbb{F}_q$. Also, for any integers $a,b$ with $a\le b$, we use the symbol $[a,b]$ to denote the set $\{a,a+1,\cdots,b\}$. 
	
	We begin with the following result.

	\begin{lemma}\label{Lem. evaluation of a product of sj-si}
		Let $q=p^f=2n+1\ge 5$ be an odd prime power with $p$ prime and $f\in\mathbb{Z}^+$. Then 
		\begin{equation*}
			(-1)^{\frac{(n-1)(n-2)}{2}}\prod_{2\le i<j\le n}\left(s_j-s_i\right)^2=\left(-\frac{1}{2}\right)^{n-2}\in\mathbb{F}_p.
		\end{equation*}
	\end{lemma}
	
	\begin{proof}
		Let 
		\begin{equation*}
			F(t)=\frac{t^n-1}{t-1}=t^{n-1}+t^{n-2}+\cdots+t+1
		\end{equation*}
	be a polynomial over $\mathbb{F}_q$. Note that $x$ is a nonzero square over $\mathbb{F}_q$ if and only if $x^n=1$. Hence it is easy to verify that 
	\begin{equation}\label{Eq. a in the proof of Lem. evaluation of a product of sj-si}
		F(t)=\frac{t^n-1}{t-1}=t^{n-1}+t^{n-2}+\cdots+t+1=\prod_{2\le j\le n}\left(t-s_j\right).
	\end{equation}
    By Eq. (\ref{Eq. a in the proof of Lem. evaluation of a product of sj-si}) we obtain that 
    $(-1)^{n-1}\prod_{2\le j\le n}s_j$ is equal to the constant term of $F(t)$, i.e., 
    \begin{equation}\label{Eq. b in the proof of Lem. evaluation of a product of sj-si}
    	\prod_{2\le j\le n}s_j=(-1)^{n-1}.
    \end{equation} 
     Also, clearly $(-1)^{n-1}\prod_{2\le j\le n}(s_j-1)$ is equal to the constant term of $F(t+1)$, that is,
     \begin{equation}\label{Eq. c in the proof of Lem. evaluation of a product of sj-si}
     	\prod_{2\le j\le n}(s_j-1)=(-1)^{n-1}\cdot n.
     \end{equation} 

     Next we consider the product 
       $$(-1)^{\frac{(n-1)(n-2)}{2}}\prod_{2\le i<j\le n}\left(s_j-s_i\right)^2.$$
     Let $F'(t)$ be the formal derivative of $F(t)$. Then one can verify that 
     \begin{align}\label{Eq. d in the proof of Lem. evaluation of a product of sj-si}
       (-1)^{\frac{(n-1)(n-2)}{2}}\prod_{2\le i<j\le n}\left(s_j-s_i\right)^2
     &=\prod_{2\le i\neq j\le n}\left(s_j-s_i\right) \notag\\
     &=\prod_{2\le j\le n}\prod_{i\in[2,n]\setminus\{j\}}\left(s_j-s_i\right) \notag\\
     &=\prod_{2\le j\le n}F'(s_j).
     \end{align}

    As $(t-1)F(t)=t^n-1$, we have 
    \begin{equation*}
    	F(t)+(t-1)F'(t)=nt^{n-1},
    \end{equation*}
     and hence 
     \begin{equation*}
     	F'(s_j)=\frac{ns_j^{n-1}}{s_j-1}=\frac{n}{s_j(s_j-1)}
     \end{equation*}
     for any $j\in[2,n]$. By the above results, 
     \begin{equation}\label{Eq. e in the proof of Lem. evaluation of a product of sj-si}
     	(-1)^{\frac{(n-1)(n-2)}{2}}\prod_{2\le i<j\le n}\left(s_j-s_i\right)^2=\prod_{2\le j\le n}\frac{n}{s_j(s_j-1)}.
     \end{equation}
     Combining Eq. (\ref{Eq. e in the proof of Lem. evaluation of a product of sj-si}) with Eq. (\ref{Eq. b in the proof of Lem. evaluation of a product of sj-si}) and Eq. (\ref{Eq. c in the proof of Lem. evaluation of a product of sj-si}), we obtain 
     \begin{equation*}
     	(-1)^{\frac{(n-1)(n-2)}{2}}\prod_{2\le i<j\le n}\left(s_j-s_i\right)^2=n^{n-2}=\left(-\frac{1}{2}\right)^{n-2}\in\mathbb{F}_p.
     \end{equation*}
    This completes the proof.
	\end{proof}
	
	Before the statement of next lemma, we introduce the following notations. Let $l$ be a positive integer and let $t_1,t_2,\cdots,t_l$ be variables. Then for any $k\in [1,l]$, the $k$-th elementary symmetric polynomial of $t_1,t_2,\cdots,t_l$ is defined by 
	$$\sigma_k(t_1,t_2,\cdots,t_l)=\sum_{1\le i_1<i_2<\cdots<i_k\le l}\prod_{1\le j\le k}t_{i_j}.$$
	In addition, we let 
	$$\sigma_0(t_1,t_2,\cdots,t_l)=1.$$
	
	In 2022, Grinberg, Sun and Zhao \cite[Theorem 3.1]{GSZ} obtained the following result.
	
	\begin{lemma}\label{Lem. the GSZ lemma}
		Let $l$ be a positive integer. Then 
		\begin{equation}\label{Eq. the GSZ determinant formula}
			\det\left[\left(x_i+y_j\right)^l\right]_{1\le i,j\le l}
			=(-1)^{\frac{l(l-1)}{2}}\cdot\prod_{1\le i<j\le l}\left(x_j-x_i\right)\left(y_j-y_i\right)\cdot\sum_{k=0}^{l}U_k,
		\end{equation}
	where 
	\begin{equation*}
		U_k=\sigma_k\left(x_1,\cdots,x_l\right)\sigma_{l-k}\left(y_1,\cdots,y_l\right)\cdot\prod_{r\in[0,l]\setminus\{k\}}\binom{l}{r}.
	\end{equation*}
	\end{lemma}
	
	We now need the following results related to some congruences involving the Pell sequence, which were obtained by Z.-H. Sun \cite[Theorem 4.1]{SZH}, Z.-W. Sun \cite[Final Remark]{Sun-PAMS} and Z.-W. Sun \cite[Remark 3.1]{Sun-IS} respectively.
	
	\begin{lemma}\label{Lem. congruence involving Pell sequences}
		Let $p\ge7$ be a prime. Then the following congruences hold.
		\begin{equation}\label{Eq. a in Lem. congruence involving Pell sequences}
			\sum_{k=1}^{p-1}\frac{2^k}{k}\equiv \frac{2-2^p}{p}-\frac{7}{12}p^2B_{p-3}\pmod{p^3\mathbb{Z}_p},
		\end{equation}
		where $B_{p-3}$ is the $(p-3)$-th Bernoulli number.
		
		\begin{equation}\label{Eq. b in Lem. congruence involving Pell sequences}
			\sum_{k=1}^{\frac{p-1}{2}}\frac{1}{k\cdot 2^k}\equiv -2^{\frac{p+1}{2}}\cdot
			\frac{P_p-2^{\frac{p-1}{2}}}{p}\pmod{p\mathbb{Z}_p}.
		\end{equation}
		
		\begin{equation}\label{Eq. c in Lem. congruence involving Pell sequences}
			4\left(\frac{2}{p}\right)P_p\equiv 2+Q_p\pmod{p^2\mathbb{Z}}.
		\end{equation}
		
	\end{lemma}
	
	Now we are in a position to prove Theorem \ref{Thm. case m is less than n+1}.
	
	{\noindent{\bf Proof of Theorem \ref{Thm. case m is less than n+1}}.} (i) By Lemma \ref{Lem. the GSZ lemma} and Lemma \ref{Lem. evaluation of a product of sj-si} we obtain 
	\begin{equation}\label{Eq. A in the proof of Thm. case m is less than n+1}
		\det B_q(n-1)=(-1)^{\frac{(n-1)(n-2)}{2}}\prod_{2\le i<j\le n}\left(s_j-s_i\right)^2\cdot \sum_{k=0}^{n-1}W_k=\left(-\frac{1}{2}\right)^{n-2}\cdot \sum_{k=0}^{n-1}W_k,
	\end{equation}
	where 
	\begin{equation}\label{Eq. B in the proof of Thm. case m is less than n+1}
		W_k=\sigma_k\left(s_2,\cdots,s_n\right)\sigma_{n-1-k}\left(s_2,\cdots,s_n\right)\cdot\prod_{r\in[0,n-1]\setminus\{k\}}\binom{n-1}{r}
	\end{equation}
	for any $k\in[0,n-1]$. 
	
	We next consider $W_k$. By Eq. (\ref{Eq. a in the proof of Lem. evaluation of a product of sj-si}) we see that 
	\begin{equation*}
			F(t)=t^{n-1}+t^{n-2}+\cdots+t+1=\prod_{2\le j\le n}\left(t-s_j\right)=\sum_{k=0}^{n-1}(-1)^k\sigma_k\left(s_2,\cdots,s_n\right)t^{n-1-k},
	\end{equation*}
	and hence for any $k\in[0,n-1]$ we have 
	\begin{equation}\label{Eq. C in the proof of Thm. case m is less than n+1}
		\sigma_k\left(s_2,\cdots,s_n\right)=(-1)^k.
	\end{equation}
	By Eq. (\ref{Eq. B in the proof of Thm. case m is less than n+1}) and Eq. (\ref{Eq. C in the proof of Thm. case m is less than n+1}) we obtain 
	\begin{equation}\label{Eq. D in the proof of Thm. case m is less than n+1}
		W_k=(-1)^{n-1}\prod_{r\in[0,n-1]\setminus\{k\}}\binom{n-1}{r}.
	\end{equation}
	
	Suppose that $f\ge2$. As $q=p^f\ge 9$, we have 
	\begin{equation*}
		n-1=\frac{q-3}{2}=\frac{p-3}{2}\cdot 1+\frac{p-1}{2}\cdot p+\cdots+\frac{p-1}{2}\cdot p^{f-1}>\frac{p+1}{2}.
	\end{equation*}
	By the Lucas congruence, for $j\in\{(p-1)/2,(p+1)/2\}$, we have 
	\begin{equation*}
		\binom{n-1}{j}\equiv \binom{(p-3)/2}{j}\binom{(p-1)/2}{0}\cdots\binom{(p-1)/2}{0}\equiv 0\pmod {p\mathbb{Z}}.
	\end{equation*}
	Thus, 
	\begin{equation*}
		\left|\left\{r\in[0,n-1]:\ \binom{n-1}{r}\equiv 0\pmod {p\mathbb{Z}}\right\}\right|\ge2.
	\end{equation*}
	This, together with Eq. (\ref{Eq. D in the proof of Thm. case m is less than n+1}), implies that $W_k=0$ (as an element of $\mathbb{F}_q$) if $f\ge2$. Hence by Eq. (\ref{Eq. A in the proof of Thm. case m is less than n+1}) we see that $B_q(n-1)$ is singular whenever $f\ge2$. 
	
	Suppose now $f=1$. Then for any $r\in[0,n-1]$ we have 
	$$\binom{n-1}{r}\not\equiv 0\pmod{p\mathbb{Z}}.$$
	Hence by Eq. (\ref{Eq. D in the proof of Thm. case m is less than n+1}) we obtain 
	\begin{align}\label{Eq. E in the proof of Thm. case m is less than n+1}
		\sum_{k=0}^{n-1}W_k&=(-1)^{n-1}\left(\sum_{r=0}^{n-1}\binom{n-1}{r}^{-1}\right)\cdot\prod_{r\in[0,n-1]}\binom{n-1}{r}\notag\\
		                   &=(-1)^{n-1}\cdot\frac{n}{2^n}\left(\sum_{k=1}^n\frac{2^k}{k}\right) \cdot\frac{((n-1)!)^n}{(0!1!\cdots (n-1)!)^2},
	\end{align}
	where the last equality follows from Sury's identity \cite{Sury} 
	\begin{equation*}
		\sum_{r=0}^{n-1}\binom{n-1}{r}^{-1}=\frac{n}{2^n}\sum_{k=1}^n\frac{2^k}{k}
	\end{equation*}
	and 
	\begin{equation*}
		\prod_{r\in[0,n-1]}\binom{n-1}{r}=\frac{((n-1)!)^n}{(0!1!\cdots (n-1)!)^2}.
	\end{equation*}
	
	We next consider 
	\begin{equation*}
		\sum_{k=1}^n\frac{2^k}{k} \mod {p\mathbb{Z}_p}.
	\end{equation*}
	One can verify that 
	\begin{align*}
		\sum_{k=1}^{p-1}\frac{2^k}{k}&=\sum_{k=1}^{n}\frac{2^k}{k}+\sum_{k=1}^n\frac{2^{p-k}}{p-k}\\
		                             &\equiv \sum_{k=1}^{n}\frac{2^k}{k}-2\sum_{k=1}^n\frac{1}{k\cdot 2^k}\pmod{p\mathbb{Z}_p}.
	\end{align*}
    Combining this with Eq. (\ref{Eq. a in Lem. congruence involving Pell sequences}) and Eq. (\ref{Eq. b in Lem. congruence involving Pell sequences}), we obtain 
    \begin{align*}
    	\sum_{k=1}^{n}\frac{2^k}{k}&\equiv \frac{2-2^p}{p}-2^{\frac{p+3}{2}}\frac{P_p-2^{\frac{p-1}{2}}}{p}\\
    	                           &\equiv \frac{2-2^p}{p}-4\left(\frac{2}{p}\right)\frac{P_p-2^{\frac{p-1}{2}}}{p}\pmod{p\mathbb{Z}_p}.
    \end{align*}
     By this and Eq. (\ref{Eq. c in Lem. congruence involving Pell sequences}), we obtain 
     \begin{equation}\label{Eq. F in the proof of Thm. case m is less than n+1}
     	\sum_{k=1}^{n}\frac{2^k}{k}\equiv \frac{-Q_p-2^p+(\frac{2}{p})2^{\frac{p+3}{2}}}{p}\equiv \frac{2-Q_p}{p}\pmod {p\mathbb{Z}_p},
     \end{equation}
     where the last congruence follows from 
     \begin{equation*}
     	2-\left(-2^p+\left(\frac{2}{p}\right)2^{\frac{p+3}{2}}\right)=2\left(2^{\frac{p-1}{2}}-\left(\frac{2}{p}\right)\right)^2\equiv 0\pmod {p^2\mathbb{Z}}.
     \end{equation*}

     Now combining Eq. (\ref{Eq. F in the proof of Thm. case m is less than n+1}) with Eq. (\ref{Eq. E in the proof of Thm. case m is less than n+1}) and Eq. (\ref{Eq. A in the proof of Thm. case m is less than n+1}), one can verify that 
     \begin{equation*}
     	\det B_p(n-1)=\frac{2\cdot((n-1)!)^n}{(0!1!\cdots (n-1)!)^2}\cdot a_p\in\mathbb{F}_p,
     \end{equation*}
     where 
     \begin{equation*}
     	a_p=\frac{2-Q_p}{p}\mod {p\mathbb{Z}}\in\mathbb{Z}/p\mathbb{Z}=\mathbb{F}_p.
     \end{equation*}
     This completes the proof of (i).
     
      (ii) Let $h(t)=t^{n-2}$. Then by Eq. (\ref{Eq. the K determinant formula}) and Lemma \ref{Lem. evaluation of a product of sj-si} we obtain 
      \begin{align}\label{Eq. a in the proof of of ii of Thm. case m is less than n+1}
      	\det B_q(n-2)&=\prod_{r=0}^{n-2}\binom{n-2}{r}\cdot \prod_{2\le i<j\le n}\left(-\left(s_j-s_i\right)^2\right) \notag\\
      	             &=\prod_{r=0}^{n-2}\binom{n-2}{r}\cdot (-1)^{\frac{(n-1)(n-2)}{2}}\prod_{2\le i<j\le n}\left(s_j-s_i\right)^2 \notag\\
      	             &=\left(-\frac{1}{2}\right)^{n-2}\cdot\prod_{r=0}^{n-2}\binom{n-2}{r}.
      \end{align}
      
      Suppose first that $f\ge2$. Noting that $q\ge9$, we obtain 
      \begin{equation*}
      	n-2=\frac{q-5}{2}=\frac{p-5}{2}\cdot 1+\frac{p-1}{2}\cdot p+\cdots+\frac{p-1}{2}\cdot p^{f-1}>\frac{p-1}{2}.
      \end{equation*}
      By the Lucas congruence again, we clearly have 
      \begin{equation*}
      	\binom{n-2}{(p-1)/2}\equiv \binom{(p-5)/2}{(p-1)/2} \binom{(p-1)/2}{0}\cdots\binom{(p-1)/2}{0}\equiv 0\pmod {p\mathbb{Z}}.
      \end{equation*}
      By this and Eq. (\ref{Eq. a in the proof of of ii of Thm. case m is less than n+1}), we see that $\det B_q(n-2)=0$ whenever $f\ge2$. 
      
      Suppose now $f=1$. Then $\binom{n-2}{r}\not\equiv 0\pmod {p\mathbb{Z}}$ for any $r\in[0,n-2]$. Hence by Eq. (\ref{Eq. a in the proof of of ii of Thm. case m is less than n+1}) again we obtain 
      \begin{equation*}
      	\det B_p(n-2)=\left(-\frac{1}{2}\right)^{n-2}\cdot\frac{((n-2)!)^{n-1}}{(0!1!\cdots (n-2)!)^2}\in\mathbb{F}_p^{\times}.
      \end{equation*}
      
      In view of the above, we have completed the proof.\qed 
      
      We next prove Corollary \ref{Cor. of Thm. case m is less than n+1}.

      {\noindent{\bf Proof of Corollary \ref{Cor. of Thm. case m is less than n+1}}.} (i) If $p\equiv 1\pmod 4$, then $n$ is even. Thus, 
      \begin{equation*}
      	\left(\frac{\det B_p(n-1)}{p}\right)=\left(\frac{2a_p}{p}\right)\left(\frac{(n-1)!}{p}\right)^n=\left(\frac{2a_p}{p}\right).
      \end{equation*}
      
      Suppose now $p\equiv 3\pmod 4$ and $p>3$. Then $n$ is odd and hence 
      \begin{equation*}
      	\left(\frac{\det B_p(n-1)}{p}\right)=\left(\frac{2a_p}{p}\right)\left(\frac{(n-1)!}{p}\right)=\left(\frac{-a_p}{p}\right)\left(\frac{n!}{p}\right)=(-1)^{\frac{h(-p)-1}{2}}\left(\frac{a_p}{p}\right).
      \end{equation*}
       The last equality follows from the Mordell congruence \cite{Mordell}
       \begin{equation*}
       	n!=\frac{p-1}{2}!\equiv (-1)^{\frac{h(-p)+1}{2}} \pmod {p\mathbb{Z}},
        \end{equation*}
         where $h(-p)$ is the class number of the imaginary quadratic field $\mathbb{Q}(\sqrt{-p})$. This completes the proof of (i).
         
         (ii) Suppose $p\equiv 1\pmod 4$. Then by the Wilson congruence, one can very that 
         \begin{equation*}
         	-1\equiv (p-1)!\equiv (-1)^n\cdot (n!)^2\equiv (n!)^2\pmod {p\mathbb{Z}}.
         \end{equation*}
         This implies $n!=\sqrt{-1}$ over $\mathbb{F}_p$. Thus,  
         \begin{equation*}
         	\left(\frac{n!}{p}\right)\equiv (n!)^n\equiv (-1)^{n/2}\equiv \left(\frac{2}{p}\right) \pmod{p\mathbb{Z}},
         \end{equation*}
         that is, 
         \begin{equation*}
         	\left(\frac{n!}{p}\right)= \left(\frac{2}{p}\right). 
         \end{equation*}
     
         Now by this and Theorem \ref{Thm. case m is less than n+1}(ii), and noting that $n$ is even in this case, we obtain 
         \begin{equation*}
         	\left(\frac{\det B_p(n-2)}{p}\right)=\left(\frac{(n-2)!}{p}\right)
         	=\left(\frac{n!}{p}\right)\left(\frac{(n-1)n}{p}\right)
         	=\left(\frac{6}{p}\right),
         \end{equation*}
         where the last equality follows from $n(n-1)=3/4$ over $\mathbb{F}_p$. 
     
         Suppose now $p\equiv 3\pmod 4$. Then $n$ is odd. By Theorem \ref{Thm. case m is less than n+1}(ii) we clearly have 
         \begin{equation*}
         	\left(\frac{\det B_p(n-2)}{p}\right)=\left(\frac{-2}{p}\right).
         \end{equation*}
     
         In view of the above, we have completed the proof. \qed

        \section{Preparations for the proof of Theorem \ref{Thm. case m=n}}
        \setcounter{lemma}{0}
        \setcounter{theorem}{0}
        \setcounter{equation}{0}
        \setcounter{conjecture}{0}
        \setcounter{remark}{0}
        \setcounter{corollary}{0}
        \subsection{A lemma on almost circulant matrices}
        Let $n\ge2$ be an integer and let $\v=(a_0,a_1,\cdots,a_{n-1})\in\mathbb{C}^n$. The circulant matrix of $\v$ is an $n\times n$ matrix defined by 
        $$C_n(\v)=[a_{j-i}]_{0\le i,j\le n-1},$$
        where $a_s=a_t$ whenever $s\equiv t\pmod n$, that is, 
        $$C_n(\v)=\left[\begin{array}{ccccc}
        	a_0       &  a_1     & \cdots  & a_{n-2}  & a_{n-1} \\
        	a_{n-1}   &  a_0     & \cdots  & a_{n-3}  & a_{n-2} \\
        	\vdots    &  \vdots  & \ddots  & \vdots   & \vdots  \\
        	a_2       &  a_3     & \cdots  & a_0      & a_1     \\
        	a_1       &  a_2      & \cdots  & a_{n-1}  & a_0
        \end{array}\right].$$
        
        Recently, the first author and the second author \cite[Theorem 4.1]{Wu-Wang} defined the almost circulant matrix $W_n(\v)$ of $\v$ by 
        \begin{equation*}
        	W_n(\v)=\left[a_{j-i}\right]_{1\le i,j\le n-1},
        \end{equation*}
        and obtained the following result.
        
        \begin{lemma}\label{Lem. determinant formula on almost circulant matrices}
        	Let $\lambda_0,\lambda_1,\cdots,\lambda_{n-1}$ be all the eigenvalues of $C_n(\v)$. Then 
        	$$\det W_n(\v)=\frac{1}{n}\sum_{l=0}^{n-1}\prod_{k\in[0,n-1]\setminus\{l\}}\lambda_k.$$
        \end{lemma}
        
        \subsection{Some $p$-adic preparations}
        Let $\pi\in\mathbb{C}_p$ with $\pi^{p-1}+p=0$ and let $\zeta_{\pi}\in\mathbb{C}_p$ be a primitive $p$-th root of unity with $\zeta_{\pi}\equiv 1+\pi\pmod{\pi^2}$. 
        
         Recall that $q=p^f$. Let $\zeta_{q-1}\in\mathbb{C}_p$ be a primitive $(q-1)$-th root of unity. Then it is known that $\mathbb{Q}_p(\zeta_{q-1})/\mathbb{Q}_p$ is an unramified extension and $[\mathbb{Q}_p(\zeta_{q-1}):\mathbb{Q}_p]=f$. Hence 
        $$\mathbb{Z}_p[\zeta_{q-1}]/\mathfrak{p}\cong\mathbb{F}_q,$$
        where $\mathfrak{p}=p\mathbb{Z}_p[\zeta_{q-1}]$ is a prime ideal of $\mathbb{Z}[\zeta_{q-1}]$. From now on, we identify $\mathbb{F}_q$ with $\mathbb{Z}_p[\zeta_{q-1}]/\mathfrak{p}$. The Teich\"{u}muller character $\omega_q: \mathbb{F}_q\rightarrow \mathbb{C}_p$ is a multiplicative character of $\mathbb{F}_q$ defined by 
        \begin{equation}\label{Eq. definition of the T character}
        	\omega_q(x\mod \mathfrak{p})\equiv x\pmod{\mathfrak{p}}
        \end{equation}
        for any $x\in \mathbb{Z}_p[\zeta_{q-1}]$. Also, it is easy to verify that $\omega_q$ is a generator of $\widehat{\mathbb{F}_q^{\times}}$. For any integer $r\in[0,q-2]$, letting 
        $$r=r_0\cdot 1+r_1\cdot p+\cdots+r_{f-1}\cdot p^{f-1} $$
        be the decomposition of $r$ in the base $p$, we define 
        \begin{equation}\label{Eq. definition of s(r)}
        	s(r)=\sum_{i=0}^{f-1}r_i.
        \end{equation}

        For $s(r)$, we have the following known result (see \cite[Lemma 3.6.7]{Cohen}).
        
        \begin{lemma}\label{Lem. evalutation of s(r)}
        	For any $r\in[0,q-2]$ we have 
        	\begin{equation*}
        		s(r)=(p-1)\sum_{i=0}^{f-1}\left\{\frac{rp^i}{q-1}\right\},
        	\end{equation*}
        where $\{x\}$ denotes the fractional part of a real number $x$. 
        \end{lemma}
        
        Recall that $n=(q-1)/2=(p^f-1)/2$. Using Lemma \ref{Lem. evalutation of s(r)}, we obtain the following result.
        
        \begin{lemma}\label{Lem. s(n)+s(n+r) is greater than s(r)}
        	Suppose $r\in [0,n-1]$. Then 
        	\begin{equation*}
        		s(n)+s(n+r)>s(r).
        	\end{equation*}
        \end{lemma}
        
        \begin{proof}
        	For any $i\in[0,f-1]$, we set 
        	$$x_i(r):=\left\{\frac{rp^i}{q-1}\right\}.$$      
        	Then it is easy to verify that 
        	\begin{equation}\label{Eq. a in the proof of Lem. s(n)+s(n+r) is greater than s(r)}
        		\frac{1}{2}+\left\{\frac{1}{2}+x_i(r)\right\}=\begin{cases}
        			x_i(r)+1 & \mbox{if}\ 0\le x_i(r)<1/2,\\
        			x_i(r)   & \mbox{if}\ 1/2\le x_i(r)<1.
        		     \end{cases}
        	\end{equation}
        	As $0\le r\le n-1$, we have $x_0(r)<1/2$. Hence by Eq. (\ref{Eq. a in the proof of Lem. s(n)+s(n+r) is greater than s(r)}) we obtain 
        	\begin{equation*}
        		\frac{1}{2}+\left\{\frac{1}{2}+x_0(r)\right\}>x_0(r).
        	\end{equation*}
        	By this, Lemma \ref{Lem. evalutation of s(r)} and Eq. (\ref{Eq. a in the proof of Lem. s(n)+s(n+r) is greater than s(r)}), one can verify that 
        	\begin{align*}
        		s(n)+s(n+r)
        		&=(p-1)\sum_{i=0}^{f-1}\left(\frac{1}{2}+\left\{\frac{(n+r)p^i}{q-1}\right\}\right)\\
        		&=(p-1)\sum_{i=0}^{f-1}\left(\frac{1}{2}+\left\{\frac{p^i}{2}+\frac{rp^i}{q-1}\right\}\right)\\
        		&=(p-1)\sum_{i=0}^{f-1}\left(\frac{1}{2}+\left\{\frac{p^i-1}{2}+\frac{1}{2}+\frac{rp^i}{q-1}\right\}\right)\\
        		&=(p-1)\sum_{i=0}^{f-1}\left(\frac{1}{2}+\left\{\frac{1}{2}+x_i(r)\right\}\right)\\
        		&>(p-1)\sum_{i=0}^{f-1}x_i(r)\\
        		&=s(r).
        	\end{align*}
        	This completes the proof.	 
        \end{proof}
        
        Recall that $p$ is an odd prime. We next introduce the $p$-adic Gamma function $\Gamma_p: \mathbb{Z}_p\rightarrow \mathbb{Z}_p^{\times}$, where $\mathbb{Z}_p^{\times}$ denotes the group of all $p$-adic units. For any $n\in\mathbb{Z}^+$, we define 
        $$\Gamma_p(n)=(-1)^n\prod_{k\in[1,n-1]\cap \mathbb{Z}_p^{\times}}k.$$
        Since $\mathbb{Z}^+$ is dense in $\mathbb{Z}_p$ and $\mathbb{Z}_p^{\times}$ is a closed multiplicative group, the $p$-adic Gamma function 
        $\Gamma_p: \mathbb{Z}_p \rightarrow \mathbb{Z}_p^{\times}$
        is defined by 
        $$\Gamma_p(x)=\lim_{i\rightarrow \infty}\Gamma_p(x_i),$$
         where $\{x_i\}_{i=1}^{\infty}$ is a sequence of positive integers $x_i$ with $\lim_{i\rightarrow \infty}x_i=x$.
        
        In this paper, we need following result on $\Gamma_p$. 
        
        \begin{lemma}\label{Lem. basic properties of p adic gamma}
        	Suppose that $p\ge5$ is a prime. Let $n\in\mathbb{Z}^+$. Then for any $x,y\in\mathbb{Z}_p$ we have 
        	$$x\equiv y\pmod{p^n\mathbb{Z}_p}\Rightarrow \Gamma_p(x) \equiv \Gamma_p(y) \pmod{p^n\mathbb{Z}_p}.$$

        \end{lemma}
        
        We conclude this section with the following result, which is known as the Gross-Koblitz formula \cite[Theorem 1.7]{Gross-Koblitz}.
        
        \begin{lemma}\label{Lem. the Gross-Koblitz formula}
        Let notations be as above. For any $r\in[0,q-2]$, let the Gauss sum
        	$$G_q(\omega_q^{-r})=\sum_{x\in\mathbb{F}_q}\omega_q^{-r}(x)\zeta_{\pi}^{\Tr_{\mathbb{F}_q/\mathbb{F}_p}(x)}.$$
        	Then 
        	\begin{equation*}
        		G_q(\omega_q^{-r})=-\pi^{s(r)}\cdot\prod_{i=0}^{f-1}\Gamma_p\left(\left\{\frac{rp^i}{q-1}\right\}\right),
        	\end{equation*}
        where $\Gamma_p$ is the $p$-adic Gamma function and $\{x\}$ is the fractional part of a real number $x$. 
        \end{lemma}

        \section{Proof of Theorem \ref{Thm. case m=n}}
        \setcounter{lemma}{0}
        \setcounter{theorem}{0}
        \setcounter{equation}{0}
        \setcounter{conjecture}{0}
        \setcounter{remark}{0}
        \setcounter{corollary}{0}
        
        Let notations be as in Section 3. In this section, we fix a generator $g$ of the cyclic group $\mathbb{F}_q^{\times}$. Recall that $n=(q-1)/2$. For any $i\in[0,n-1]$, we let 
        $$a_i=\omega_q^{-n}\left(1+g^{2i}\right),$$
        and let 
        $$\v=\left(a_0,a_1,\cdots,a_{n-1}\right).$$
        
        We begin with the following result.
        
        \begin{lemma}\label{Lem. eigenvalues of Cn(v)}
        	For any $r\in[0,n-1]$, let 
        	\begin{equation}\label{Eq. eigenvalues of Cn(v)}
        		\lambda_r:=\frac{(-1)^r}{2}J_q\left(\omega_q^{-n},\omega_q^{-r}\right)+\frac{(-1)^{n+r}}{2}J_q\left(\omega_q^{-n},\omega_q^{-(n+r)}\right).
        	\end{equation}
        Then $\lambda_0=-1,\lambda_1,\cdots,\lambda_{n-1}$ are exactly all the eigenvalues of the circulant matrix $C_n(\v)$. 
        \end{lemma}

        \begin{proof}
        	For any $r\in[0,n-1]$, one can verify that 
        	\begin{align*}
        		&\sum_{0\le j\le n-1}\omega_q^{-n}\left(1+g^{2j-2i}\right)\omega_q^{-r}\left(g^{2j}\right)\\
        	   =&\sum_{0\le j\le n-1}\omega_q^{-n}\left(1+g^{2j-2i}\right)\omega_q^{-r}\left(g^{2j-2i}\right)\omega_q^{-r}\left(g^{2i}\right)\\
        	   =&\sum_{0\le j\le n-1}\omega_q^{-n}\left(1+g^{2j}\right)\omega_q^{-r}\left(g^{2j}\right)\omega_q^{-r}\left(g^{2i}\right).
        	\end{align*}
            This implies that for any $r\in[0,n-1]$, 
            \begin{equation*}
            	C_n(\v)\u_r=y_r\u_r,
            \end{equation*}
            where the column vector 
            $$\u_r=\left(\omega_q^{-r}(g^{0}),\omega_q^{-r}(g^2),\cdots,\omega_q^{-r}(g^{2(n-1)})\right)^T,$$
            and 
            $$y_r=\sum_{0\le j\le n-1}\omega_q^{-n}\left(1+g^{2j}\right)\omega_q^{-r}\left(g^{2j}\right).$$
         As $\u_0,\u_1,\cdots,\u_{n-1}$ are clearly linearly independent over $\mathbb{C}$, the numbers $y_0,y_1,\cdots,y_{n-1}$ are all the eigenvalues of $C_n(\v)$. Note that 
         $$\frac{1}{2}\left(\varepsilon(x)+\omega_q^{-n}(x)\right)=\begin{cases}
         	1 & \mbox{if}\ x\in\{s_1,s_2,\cdots,s_n\},\\
         	0 & \mbox{otherwise}.
         \end{cases}$$
         Thus, for each $r\in[0,n-1]$, we have 
         \begin{align*}
         	y_r&=\sum_{0\le j\le n-1}\omega_q^{-n}\left(1+g^{2j}\right)\omega_q^{-r}\left(g^{2j}\right)\\
         	   &=\sum_{1\le j\le n}\omega_q^{-n}\left(1+s_j\right)\omega_q^{-r}\left(s_j\right)\\
         	   &=\frac{1}{2}\sum_{x\in\mathbb{F}_q}\left(\varepsilon(x)+\omega_q^{-n}(x)\right)\omega_q^{-n}\left(1+x\right)\omega_q^{-r}\left(x\right)\\
         	   &=\frac{(-1)^r}{2}\sum_{x\in\mathbb{F}_q}\omega_q^{-n}\left(1+x\right)\omega_q^{-r}\left(-x\right)+\frac{(-1)^{n+r}}{2}\sum_{x\in\mathbb{F}_q}\omega_q^{-n}\left(1+x\right)\omega_q^{-(n+r)}\left(-x\right)\\
         	   &=\lambda_r.
         \end{align*}
        
        For $\lambda_0$, one can verify that 
        \begin{align*}
        	\lambda_0&=\frac{1}{2}J_q\left(\omega_q^{-n},\varepsilon\right)+\frac{(-1)^{n}}{2}J_q\left(\omega_q^{-n},\omega_q^{{-n}}\right)\\
        	&=\frac{1}{2}J_q\left(\omega_q^{-n},\varepsilon\right)+\frac{(-1)^{n}}{2}J_q\left(\omega_q^{-n},\omega_q^{{n}}\right)\\
        	&=-\frac{1}{2}\omega_q(1)+\frac{(-1)^{n}}{2}(-\omega_q^{-n}(-1))\\
        	&=-1.
        \end{align*}
        
        In view of the above, we have completed the proof.
        \end{proof}
        
        Now we are in a position to prove Theorem \ref{Thm. case m=n}.
        
        {\noindent{\bf Proof of Theorem \ref{Thm. case m=n}}.} Recall that $g$ is a generator of $\mathbb{F}_q^{\times}$. Then it is clear that 
        \begin{equation*}
        	\det B_q(n)=\det\left[\left(1+\frac{s_j}{s_i}\right)^n\right]_{2\le i,j\le n}=\det\left[\left(1+g^{2j-2i}\right)^n\right]_{1\le i,j\le n-1}.
        \end{equation*}
        By the definition (\ref{Eq. definition of the T character}) of the Teich\"{u}muller character $\omega_q$ of $\mathbb{F}_q$, we see that 
        \begin{equation}\label{Eq. A in the proof of Thm. case m=n}
        	\det B_q(n)=\det \left[\omega_q^{-n}\left(1+g^{2j-2i}\right)\right]_{1\le i,j\le n-1}\mod{\mathfrak{p}}=\det W_n(\v) \mod{\mathfrak{p}},
        \end{equation}
        where 
        $$\v=\left(\omega_q^{-n}\left(1+g^{0}\right),\omega_q^{-n}\left(1+g^{2}\right),\cdots,\omega_q^{-n}\left(1+g^{2(n-1)}\right)\right),$$
        and $\mathfrak{p}=p\mathbb{Z}_p[\zeta_{q-1}]$ is a prime ideal. Hence we focus on the determinant of the almost circulant matrix $W_n(\v)$. By Lemma \ref{Lem. eigenvalues of Cn(v)} and Lemma \ref{Lem. determinant formula on almost circulant matrices} we have 
        \begin{equation}\label{Eq. B in the proof of Thm. case m=n}
        	\det B_q(n)=\det W_n(\v) \mod {\mathfrak{p}}=\frac{1}{n}\sum_{l=0}^{n-1}\prod_{r\in[0,n-1]\setminus\{l\}}\lambda_r \mod {\mathfrak{p}},
        \end{equation}
        where $\lambda_r$ is defined by Eq. (\ref{Eq. eigenvalues of Cn(v)}). We next consider $\lambda_r\mod{\mathfrak{p}}$. Suppose $r\in[1,n-1]$. Then 
        \begin{equation*}
        	J_q\left(\omega_q^{-n},\omega_q^{-(n+r)}\right)=\frac{G_q(\omega_q^{-n})G_q(\omega_q^{-(n+r)})}{G_q(\omega_q^{-r})}.
        \end{equation*}
        By Lemma \ref{Lem. the Gross-Koblitz formula} and Lemma \ref{Lem. s(n)+s(n+r) is greater than s(r)}, we obtain 
        \begin{equation*}
        	\ord_p\left(J_q\left(\omega_q^{-n},\omega_q^{-(n+r)}\right)\right)=\frac{1}{p-1}\left(s(n)+s(n+r)-s(r)\right)>0.
        \end{equation*}
        This implies that 
        \begin{equation*}
        	J_q\left(\omega_q^{-n},\omega_q^{-(n+r)}\right)\equiv 0\pmod{\mathfrak{p}}
        \end{equation*}
        for any $r\in[1,n-1]$. Hence by Eq. (\ref{Eq. eigenvalues of Cn(v)}) we have 
        \begin{equation}\label{Eq. C in the proof of Thm. case m=n}
        	\lambda_r\equiv \frac{(-1)^r}{2}J_q\left(\omega_q^{-n},\omega_q^{-r}\right)\pmod{\mathfrak{p}}
        \end{equation}
        whenever $r\in[1,n-1]$. 
        
        {\bf Case I:} $f\ge2$. 
        
        Suppose $r\in[1,n-1]$. Then 
        $$J_q\left(\omega_q^{-n},\omega_q^{-r}\right)=\frac{G_q(\omega_q^{-n})G_q(\omega_q^{-r})}{G_q(\omega_q^{-(n+r)})}.$$
        By Lemma \ref{Lem. the Gross-Koblitz formula} we obtain 
        \begin{equation}\label{Eq. D in the proof of Thm. case m=n}
        	\ord_p\left(J_q\left(\omega_q^{-n},\omega_q^{-r}\right)\right)=\frac{1}{p-1}\left(s(n)+s(r)-s(n+r)\right).
        \end{equation}
        
        As $q=p^f\ge9$, we have $n-1\ge (p+3)/2$. By this, for $r\in\{(p+1)/2,(p+3)/2\}\subseteq[1,n-1]$, we have 
        $$s(n)+s(r)>s(n+r).$$
        This, together with Eq (\ref{Eq. D in the proof of Thm. case m=n}), implies that 
        $$\left|\left\{1\le r\le n-1: J_q\left(\omega_q^{-n},\omega_q^{-r}\right)\equiv 0\pmod{\mathfrak{p}}\right\}\right|=\left|\left\{1\le r\le n-1: \lambda_r\equiv 0\pmod{\mathfrak{p}}\right\}\right|\ge2.$$
        By this, Eq. (\ref{Eq. C in the proof of Thm. case m=n}) and Eq. (\ref{Eq. B in the proof of Thm. case m=n}), we see that $B_q(n)$ is singular if $f\ge2$. 
        
        {\bf Case II:} $f=1$. 
        
        Note first that, in this case, $\mathbb{Q}_p(\zeta_{q-1})=\mathbb{Q}_p$ and $\mathfrak{p}=p\mathbb{Z}_p$. By Eq. (\ref{Eq. D in the proof of Thm. case m=n}), for any $r\in[1,n-1]$, we see that 
        $$\ord_p\left(J_p\left(\omega_p^{-n},\omega_p^{-r}\right)\right)=\frac{1}{p-1}\left(s(n)+s(r)-s(n+r)\right)=\frac{1}{p-1}(n+r-(n+r))=0.$$
        Thus, for any $r\in[1,n-1]$, 
        $$\lambda_r\equiv \frac{(-1)^r}{2}J_q\left(\omega_p^{-n},\omega_p^{-r}\right)\not\equiv0\pmod{p\mathbb{Z}_p},$$
        that is, $\lambda_r$ is a $p$-adic unit. By this and Eq. (\ref{Eq. B in the proof of Thm. case m=n}), we obtain 
        \begin{equation}\label{Eq. E in the proof of Thm. case m=n}
         \det B_q(n)=-2\lambda_0\lambda_1\cdots\lambda_{n-1}\left(\frac{1}{\lambda_0}+\frac{1}{\lambda_1}+\cdots+\frac{1}{\lambda_{n-1}}\right) \mod{p\mathbb{Z}_p}.
        \end{equation}
        
        We next consider $\lambda_r\mod{p\mathbb{Z}_p}$. Suppose $r\in[1,n-1]$. Then By Lemma \ref{Lem. the Gross-Koblitz formula} and Lemma \ref{Lem. basic properties of p adic gamma} one can verify that  
        \begin{align*}
        	\lambda_r&\equiv \frac{(-1)^r}{2}J_p\left(\omega_p^{-n},\omega_p^{-r}\right)\\
        	         &\equiv \frac{(-1)^r}{2}\frac{G_p(\omega_p^{-n})G_p(\omega_p^{-r})}{G_p(\omega_p^{-(n+r)})}\\
        	         &\equiv \frac{(-1)^{r+1}}{2}\cdot\frac{\Gamma_p\left(\frac{n}{p-1}\right)\Gamma_p\left(\frac{r}{p-1}\right)}{\Gamma_p\left(\frac{n+r}{p-1}\right)}\\
        	         &\equiv \frac{(-1)^{r+1}}{2}\cdot\frac{\Gamma_p\left(p-n\right)\Gamma_p\left(p-r\right)}{\Gamma_p\left(p-n-r\right)}\\
        	         &\equiv \frac{(-1)^{r}}{2}\frac{(p-n-1)!\cdot (p-r-1)!}{(p-1-n-r)!}\\
        	         &\equiv \frac{(-1)^{r+1}}{2}\frac{(n+r)!}{n!\cdot r!}\\
        	         &\equiv \frac{(-1)^{r+1}}{2}\binom{n+r}{r}\\
        	         &\equiv -\frac{1}{2}\binom{n}{r}\pmod{p\mathbb{Z}_p}.
        \end{align*}
        The last congruence follows from 
        $$(-1)^r\cdot\binom{n+r}{r}\equiv \binom{n}{r}\pmod{p\mathbb{Z}}$$
        for any $r\in[1,n-1]$. 
        
        By the above results and Sury's identity \cite{Sury}, we obtain 
        \begin{align*}
        	  &-2\lambda_0\lambda_1\cdots\lambda_{n-1}\left(\frac{1}{\lambda_0}+\frac{1}{\lambda_1}+\cdots+\frac{1}{\lambda_{n-1}}\right)\\
        \equiv&	(-1)^{n-1}\left(\frac{1}{2}\right)^{n-2}\cdot\prod_{r=0}^n\binom{n}{r}\cdot\left(3-2\sum_{r=0}^{n}\binom{n}{r}^{-1}\right)\\
        \equiv& (-1)^{n-1}\left(\frac{1}{2}\right)^{n-2}\cdot\frac{(n!)^{n+1}}{(0!1!\cdots n!)^2}\cdot \left(3-\frac{n+1}{2^n}\sum_{k=1}^{n+1}\frac{2^k}{k}\right)\\
        \equiv& (-1)^{n-1}\left(\frac{1}{2}\right)^{n-2}\cdot\frac{(n!)^{n+1}}{(0!1!\cdots n!)^2}\cdot \left(3-\frac{1}{2^{n+1}}\sum_{k=1}^{n+1}\frac{2^k}{k}\right)\\
        \equiv& (-1)^{n-1}\left(\frac{1}{2}\right)^{n-2}\cdot\frac{(n!)^{n+1}}{(0!1!\cdots n!)^2}\cdot \left(1-\frac{1}{2^{n+1}}\sum_{k=1}^{n}\frac{2^k}{k}\right)\pmod{p\mathbb{Z}_p}.
        \end{align*}
        
        By this and using Eq. (\ref{Eq. F in the proof of Thm. case m is less than n+1}) and Eq. (\ref{Eq. c in Lem. congruence involving Pell sequences}), we finally obtain 
        \begin{align*}
        	      &-2\lambda_0\lambda_1\cdots\lambda_{n-1}\left(\frac{1}{\lambda_0}+\frac{1}{\lambda_1}+\cdots+\frac{1}{\lambda_{n-1}}\right)\\
        	\equiv& 
        	(-1)^{n}\left(\frac{1}{2}\right)^{n-2}\cdot\frac{(n!)^{n+1}}{(0!1!\cdots n!)^2}\cdot\frac{2(\frac{2}{p})-2P_p-p}{p}\pmod{p\mathbb{Z}_p}.
        \end{align*}
        
        In view of the above, we have completed the proof. \qed

        \section{Proof of Theorem \ref{Thm. a variant of Carlitz}}
        \setcounter{lemma}{0}
        \setcounter{theorem}{0}
        \setcounter{equation}{0}
        \setcounter{conjecture}{0}
        \setcounter{remark}{0}
        \setcounter{corollary}{0}
        
        We begin with the following lemma.
        
        \begin{lemma}\label{Lem. product of Jacobi sums}
        	Let $\chi_q$ be a generator of $\widehat{\mathbb{F}_q^{\times}}$. Then for any nontrivial character $\psi\in\widehat{\mathbb{F}_q^{\times}}$, the following results hold.
        	\begin{equation}\label{Eq. product of Jacobi sums}
        		A_q =\prod_{r=0}^{q-2}J_q(\psi,\chi_q^r)=\frac{G_q(\psi)^{q-1}}{q}.
        	\end{equation}
           
           \begin{equation}\label{Eq. sum of reciprocal of Jacobi sums}
           	S_q =\sum_{r=0}^{q-2}\frac{1}{J_q(\psi,\chi_q^r)}=\frac{1-q}{q}\left(1+\psi(-1)\right).
           \end{equation}
        
           \begin{equation}\label{Eq. sum of (-1) r reciprocal of Jacobi sums}
           		T_q =\sum_{r=0}^{q-2}\frac{(-1)^r}{J_q(\psi,\chi_q^r)}=\frac{1-q}{q}(2-\overline{\psi(2)}). 
           \end{equation}
        \end{lemma}
        
        \begin{proof}
        	Since $\chi_q$ is a generator of $\widehat{\mathbb{F}_q^{\times}}$ and $\psi\neq\varepsilon$, there is a unique integer $j\in[1,q-2]$ such that $\chi_q^{-j}=\psi$. One can verify that 
        	\begin{align*}
        		\prod_{r=0}^{q-2}J_q(\psi,\chi_q^r)
        		&=\prod_{r=0}^{q-2}J_q(\chi_q^{-j},\chi_q^r)\\
        		&=J_q(\chi_q^{-j},\chi_q^j)\cdot\prod_{r\in[0,q-2]\setminus\{j\}}J_q(\chi_q^{-j},\chi_q^r)\\
        		&=(-1)^{j+1}\cdot\prod_{r\in[0,q-2]\setminus\{j\}}\frac{G_q(\chi_q^{-j})G_q(\chi_q^r)}{G_q(\chi_q^{-j+r})}\\
        		&=(-1)^{j+1}\cdot\frac{G_q(\varepsilon)}{G_q(\chi_q^{-j})G_q(\chi_q^j)}\cdot\prod_{r\in[0,q-2]}\frac{G_q(\chi_q^{-j})G_q(\chi_q^r)}{G_q(\chi_q^{-j+r})}\\
        		&=\frac{G_q(\chi_q^{-j})^{q-1}}{q}\cdot\prod_{r\in[0,q-2]}\frac{G_q(\chi_q^r)}{G_q(\chi_q^{-j+r})}\\
        		&=\frac{G_q(\psi)^{q-1}}{q}.
        	\end{align*}
        This completes the proof of Eq. (\ref{Eq. product of Jacobi sums}). 
        
        Now we prove Eq. (\ref{Eq. sum of reciprocal of Jacobi sums}). As $|J_q(\chi_q^{-j},\chi_q^r)|=\sqrt{q}$ whenever $r\in[1,q-2]\setminus\{j\}$, one can verify that 
        \begin{align*}
        	\sum_{r=0}^{q-2}\frac{1}{J_q(\psi,\chi_q^r)}
        	&=\sum_{r=0}^{q-2}\frac{1}{J_q(\chi_q^{-j},\chi_q^r)}\\
        	&=\frac{1}{J_q(\chi_q^{-j},\varepsilon)}+\frac{1}{J_q(\chi_q^{-j},\chi_q^{j})}+\frac{1}{q}\sum_{r\in[1,q-2]\setminus\{j\}}\frac{q}{J_q(\chi_q^{-j},\chi_q^r)}\\
        	&=-(1+(-1)^{j})+\frac{1}{q}\sum_{r\in[1,q-2]\setminus\{j\}}\overline{J_q(\chi_q^{-j},\chi_q^r)}\\
        	&=-(1+\psi(-1))+\frac{1}{q}\sum_{r\in[0,q-2]}\overline{J_q(\chi_q^{-j},\chi_q^r)}-\frac{1}{q}\overline{J_q(\psi,\varepsilon)}-\frac{1}{q}\overline{J_q(\psi,\psi^{-1})}\\
        	&=\frac{1-q}{q}\left(1+\psi(-1)\right),
        \end{align*}
    where the last equality follows from 
    \begin{equation*}
    	\sum_{r\in[0,q-2]}J_q(\chi_q^{-j},\chi_q^r)=0.
    \end{equation*}
    This completes the proof of Eq. (\ref{Eq. sum of reciprocal of Jacobi sums})

    Finally, we prove Eq. (\ref{Eq. sum of (-1) r reciprocal of Jacobi sums}). We first claim that 
   \begin{equation}\label{Eq. a in the proof of Lem. product of Jacobi sums}
   	\sum_{r\in[0,(q-3)/2]}J_q(\psi,\chi_q^{2r})=\frac{(q-1)\psi(2)}{2}.
   \end{equation}
    In fact, it is easy to verify that 
    \begin{align*}
    	\sum_{r\in[0,(q-3)/2]}J_q(\psi,\chi_q^{2r})
    	&=\sum_{r\in[0,(q-3)/2]}\sum_{x\in\mathbb{F}_q}\psi(1-x)\chi_q^{2r}(x)\\
    	&=\sum_{x\in\mathbb{F}_q}\psi(1-x)\sum_{r\in[0,(q-3)/2]}\chi_q^{2r}(x)\\
    	&=\frac{q-1}{2}\psi(2),
    \end{align*}
     where the last equality follows from 
     \begin{equation*}
     	\sum_{r\in[0,(q-3)/2]}\chi_q^{2r}(x)=\begin{cases}
     		(q-1)/2 & \mbox{if}\ x=\pm 1,\\
     		0       & \mbox{otherwise.}
     	\end{cases}
     \end{equation*}
     By Eq. (\ref{Eq. a in the proof of Lem. product of Jacobi sums}) we obtain 
    \begin{align*}
    	S_q+T_q&=2\sum_{r\in[0,(q-3)/2]}\frac{1}{J_q(\psi,\chi_q^{2r})}\\
    	       &=\frac{2}{J_q(\psi,\varepsilon)}+\frac{1+\psi(-1)}{J_q(\psi,\psi^{-1})}+\frac{2}{q}\sum_{\substack{r\in[1,(q-3)/2]\\ \chi_q^{2r}\neq \psi^{-1}}}\frac{q}{J_q(\psi,\chi_q^{2r})}\\
    	       &=-(3+\psi(-1))+\frac{2}{q}\sum_{\substack{r\in[1,(q-3)/2]\\ \chi_q^{2r}\neq \psi^{-1}}}\overline{{J_q(\psi,\chi_q^{2r})}}\\
    	       &=-(3+\psi(-1))+\frac{2}{q}\sum_{r\in[0,(q-3)/2]}\overline{{J_q(\psi,\chi_q^{2r})}}-\frac{2}{q}\overline{J_q(\psi,\varepsilon)}-\frac{1+\psi(-1)}{q}\overline{J_q(\psi,\psi^{-1})}\\
    	       &=\frac{1-q}{q}\left(3+\psi(-1)-\overline{\psi(2)}\right).    
    \end{align*}
    Combining this with Eq. (\ref{Eq. sum of reciprocal of Jacobi sums}), we have 
    \begin{equation*}
    	T_q=(S_q+T_q)-S_q=\frac{1-q}{q}(2-\overline{\psi(2)}). 
    \end{equation*}

     In view of the above, we have completed the proof.
        \end{proof}
       
     \begin{remark}
     	In 1987, Greene \cite[Definition 2.4]{Greene} used Jacobi sums to obtain a finite field analogue of binomial coefficients. In fact, for any $A,B\in\widehat{\mathbb{F}_q^{\times}}$, Green defined 
     	$$\binom{A}{B}:=\frac{B(-1)}{q}J_q(A,\overline{B}).$$
     	By Eq. (\ref{Eq. sum of (-1) r reciprocal of Jacobi sums}) for any nontrivial character $\psi$, we obtain 
     	\begin{equation}\label{Eq. finite field analogue of Sury identity}
     		\sum_{r=0}^{q-2}\binom{\psi}{\chi_q^r}^{-1}=q\sum_{r=0}^{q-2}\frac{(-1)^r}{J_q(\psi,\chi_q^{-r})}=qT_q=(1-q)(2-\overline{\psi(2)}).
     	\end{equation}
     	The identity (\ref{Eq. finite field analogue of Sury identity}) can be viewed as a finite field analogue of Sury's identity
     		\begin{equation*}
     		\sum_{r=0}^{n-1}\binom{n-1}{r}^{-1}=\frac{n}{2^n}\sum_{k=1}^n\frac{2^k}{k}.
     	\end{equation*}
     \end{remark}

        From now on, we fix a generator $g$ of $\mathbb{F}_q^{\times}$. For any nontrivial character $\psi\in\widehat{\mathbb{F}_q^{\times}}$, define the circulant matrices
        \begin{equation}\label{Eq. definition of Mq}
        	M_q(\psi):=\left[\psi\left(g^{j-i}-1\right)\right]_{0\le i,j\le q-2},
        \end{equation}
        and 
        \begin{equation}\label{Eq. definition of Nq}
        	N_q(\psi):=\left[\psi\left(g^{j-i}+1\right)\right]_{0\le i,j\le q-2}.
        \end{equation}
    We need the following lemma.
    
    \begin{lemma}\label{Lem. eigenvalues of Mq and Nq}
    	Let $\chi_q$ be a generator of $\widehat{\mathbb{F}_q^{\times}}$. Then the numbers
    	$$\alpha_r=\psi(-1)J_q(\psi,\chi_q^r)\ (r=0,1,\cdots,q-2)$$
    	are exactly all the eigenvalues of $M_q(\psi)$. Also, the numbers 
    	$$\beta_r=(-1)^rJ_q(\psi,\chi_q^r)\ (r=0,1,\cdots,q-2)$$
    	are exactly all the eigenvalues of $N_q(\psi)$.
    \end{lemma}
    
    \begin{proof}
    	For any $r\in[0,q-2]$, it is easy to see that 
    	\begin{align*}
    		\sum_{j=0}^{q-2}\psi(g^{j-i}-1)\chi_q^r(g^j)
    		&=\sum_{j=0}^{q-2}\psi(g^{j-i}-1)\chi_q^r(g^{j-i})\chi_q^r(g^i)\\
    		&=\sum_{j=0}^{q-2}\psi(g^{j}-1)\chi_q^r(g^{j})\chi_q^r(g^i)\\
    		&=\sum_{x\in\mathbb{F}_q}\psi(x-1)\chi_q^r(x)\chi_q^r(g^i)\\
    		&=\psi(-1)J_q(\psi,\chi_q^r)\chi_q^r(g^i).
    	\end{align*}
    This implies that 
    $$M_q(\psi){\bm \xi}_r=\alpha_r{\bm \xi}_r,$$
    where 
    $${\bm \xi}_r=\left(\chi_q^r(g^0),\chi_q^r(g^1),\cdots,\chi_q^r(g^{q-2})\right)^T.$$
    As these ${\bm \xi}_r$ are linearly independent over $\mathbb{C}$, the numbers $\alpha_0,\alpha_1,\cdots,\alpha_{q-2}$ are all the eigenvalues of $M_q(\psi)$. Using essentially the same method, one can easily obtain that all the eigenvalues of $N_q(\psi)$ are $\beta_0,\beta_1,\cdots,\beta_{q-2}$. 
    
    In view of the above, we have completed the proof.
    \end{proof}
    
    Now we are in a position to prove our last theorem.
    
   {\noindent{\bf Proof of Theorem \ref{Thm. a variant of Carlitz}}.} (i) It is clear that 
   \begin{equation*}
   	G(t)=\frac{t^{q-1}-1}{t-1}=1+t+\cdots+t^{q-2}=\prod_{i=2}^{q-1}\left(t-x_i\right)
   \end{equation*}
   over $\mathbb{F}_q$. This implies 
   \begin{equation}\label{Eq. A in the proof of Thm. a variant of Carlitz}
   	\prod_{i=2}^{q-2}x_i=-1.
   \end{equation}
     Recall that $g$ is a generator of $\mathbb{F}_q^{\times}$. By Eq. (\ref{Eq. A in the proof of Thm. a variant of Carlitz}) we obtain 
    \begin{equation}\label{Eq. B in the proof of Thm. a variant of Carlitz}
    	\det D_q^{-}(\psi)=\prod_{i=2}^{q-1}\psi(x_i)\cdot\det\left[\psi\left(\frac{x_j}{x_i}-1\right)\right]_{2\le i,j\le q-1}=\psi(-1)\cdot\det\left[\psi(g^{j-i}-1)\right]_{1\le i,j\le q-2}.
    \end{equation}  
       Note that  $\left[\psi(g^{j-i}-1)\right]_{1\le i,j\le q-2}$ is an almost circulant matrix.  By Lemma \ref{Lem. determinant formula on almost circulant matrices} and Lemma \ref{Lem. eigenvalues of Mq and Nq}, we obtain 
       \begin{align*}
       	\det\left[\psi(g^{j-i}-1)\right]_{1\le i,j\le q-2}
       	&=\frac{1}{q-1}\cdot\alpha_0\alpha_1\cdots\alpha_{q-2}\cdot\left(\frac{1}{\alpha_0}+\frac{1}{\alpha_1}+\cdots+\frac{1}{\alpha_{q-2}}\right)\\
       	&=\frac{\psi(-1)}{q-1}\cdot\prod_{r=0}^{q-2}J_q(\psi,\chi_q^r)\cdot\left(\sum_{r=0}^{q-2}\frac{1}{J_q(\psi,\chi_q^r)}\right)\\
       	&=-\frac{(1+\psi(-1))}{q^2}G_q(\psi)^{q-1}.
       \end{align*}
        Combining this with Eq. (\ref{Eq. B in the proof of Thm. a variant of Carlitz}), we obtain 
        $$\det D_q^{-}(\psi)=-\frac{(1+\psi(-1))}{q^2}G_q(\psi)^{q-1}.$$
        This completes the proof of (i).
        
        (ii) Similar to Eq. (\ref{Eq. B in the proof of Thm. a variant of Carlitz}), we have 
        \begin{equation}\label{Eq. C in the proof of Thm. a variant of Carlitz}
        	\det D_q^+(\psi)=\psi(-1)\cdot\det\left[\psi(g^{j-i}+1)\right]_{1\le i,j\le q-2}.
        \end{equation}
        Since $\left[\psi(g^{j-i}+1)\right]_{1\le i,j\le q-2}$ is an almost circulant matrix, by  Lemma \ref{Lem. determinant formula on almost circulant matrices} and Lemma \ref{Lem. eigenvalues of Mq and Nq}, one can verify that 
        \begin{align*}
        	\det\left[\psi(g^{j-i}+1)\right]_{1\le i,j\le q-2}
        	&=\frac{1}{q-1}\cdot\beta_0\beta_1\cdots\beta_{q-2}\cdot\left(\frac{1}{\beta_0}+\frac{1}{\beta_1}+\cdots+\frac{1}{\beta_{q-2}}\right)\\
        	&=\frac{(-1)^{(q-1)/2}}{q-1}\cdot\prod_{r=0}^{q-2}J_q(\psi,\chi_q^r)\cdot\left(\sum_{r=0}^{q-2}\frac{(-1)^r}{J_q(\psi,\chi_q^r)}\right)\\
        	&=\frac{(-1)^{(q+1)/2}}{q^2}\left(2-\overline{\psi(2)}\right)G_q(\psi)^{q-1}.
        \end{align*}
        By this and Eq. (\ref{Eq. C in the proof of Thm. a variant of Carlitz}), we finally obtain 
        $$\det D_q^{+}(\psi)=\frac{(-1)^{(q+1)/2}\cdot\psi(-1)}{q^2}\left(2-\overline{\psi(2)}\right)G_q(\psi)^{q-1}.$$

        In view of the above, we have completed the proof. \qed

	\Ack\ The authors would like to thank the referee for helpful comments. 
	We also thank Prof. Hao Pan for his helpful comments and steadfast encouragement.
	
	This research was supported by the Natural Science Foundation of China (Grant Nos. 12101321, 12201291 and 12371004). The first author was supported by the Natural Science Foundation of the Higher Education Institutions of Jiangsu Province (Grant No. 25KJB110010).


\begin{thebibliography}{99}
		
	\bibitem{Carlitz}  L. Carlitz, Some cyclotomic matrices, Acta Arith. 5 (1959), 293--308.
		
	\bibitem{Chapman} R. Chapman, Determinants of Legendre symbol matrices, Acta Arith. 115 (2004), 231--244.
		
	\bibitem{Cohen} H. Cohen, Number Theory, Vol. I. Tools and Diophantine Equations, Springer, 2007.	
		
	\bibitem{GSZ} D. Grinberg, Z.-W. Sun and L. Zhao,  Proof of three conjectures on determinants related to quadratic residues, Linear Multilinear Algebra 70 (2022), no. 19, 3734--3746.
		
	\bibitem{Gross-Koblitz} B. Gross and N. Koblitz, Gauss sums and the $p$-adic $\Gamma$-function, Ann. Math 109 (1979), 569--581. 
		
	\bibitem{K} C. Krattenthaler, Advanced determinant calculus: a complement, Linear Algebra Appl. 411 (2005), 68--166.
		
	\bibitem{Greene} J. Greene, Hypergeometric functions over finite fields, Trans. Am. Math. Soc. 301 (1987), 77--101.
	
	\bibitem{Lehmer}  D.H. Lehmer, On certain character matrices, Pac. J. Math. 6 (1956), 491--499.
		
	\bibitem{Mordell} L.J. Mordell, The congruence $((p-1)/2)!\equiv \pm1 \pmod p$, Am. Math. Mon. 68 (1961), 145--146.
		
	\bibitem{SZH} Z.-H. Sun, Congruences involving Bernoulli and Euler numbers, J. Number Theory 128 (2008), no. 2, 280--312.
		
	\bibitem{Sun-PAMS} Z.-W. Sun, A congruence for primes, Proc. Amer. Math. Soc. 123 (1995), no. 5, 1341--1346.
	
	\bibitem{Sun-IS} Z.-W. Sun, On the sum $\sum_{k\equiv r\pmod m}\binom{n}{k}$ and related congruences, Israel J. Math. 128 (2002), 135--156.
	
	\bibitem{Sun19}  Z.-W. Sun, On some determinants with Legendre symbol entries, Finite Fields Appl. 56 (2019), 285--307.
	
	\bibitem{Sun24} Z.-W. Sun, Some determinants involving quadratic residues modulo primes, preprint, arXiv:2401.14301.
	
	\bibitem{Sury} B. Sury, Sum of reciprocals of the binomial coefficients, Europ. J. Combinatorics, 14 (1993), 351--353.
	
	\bibitem{Vsemirnov12} M. Vsemirnov, On the evaluation of R. Chapman’s “evil determinant”, Linear Algebra Appl. 436 (2012), 4101--4106.
	
	\bibitem{Vsemirnov13} M. Vsemirnov, On R. Chapman’s “evil determinant”: case $p\equiv 1 \pmod4$, Acta Arith. 159 (2013), 331--344.
	
	\bibitem{Wu-CR} H.-L. Wu, Determinants concerning Legendre symbols, C. R. Math. Acad. Sci. Paris 359 (2021), 651--655. 
	
	\bibitem{Wu-Wang} H.-L. Wu and L.-Y. Wang, The Gross-Koblitz formula and almost circulant matrices related to Jacobi sums, Finite Fields Appl. 103 (2025), Article 102581.
	\end{thebibliography}
\end{document}